\documentclass[11pt,reqno]{amsart}
\usepackage{amssymb}
\usepackage{amsfonts}
\usepackage{mathrsfs}
\usepackage{amsmath}
\usepackage{graphicx}
\usepackage{hyperref}
\usepackage{float}
\usepackage{epstopdf}
\usepackage{color}
\usepackage{bm}
\usepackage{comment}
\usepackage{soul}
\usepackage{enumerate}
\usepackage[title]{appendix}

\allowdisplaybreaks
\newtheorem{theorem}{Theorem}[section]
\newtheorem{proposition}[theorem]{Proposition}
\newtheorem{lemma}[theorem]{Lemma}
\newtheorem{corollary}[theorem]{Corollary}
\newtheorem{remark}[theorem]{Remark}

\newtheorem{definition}[theorem]{Definition}

\def\mcD{\mathcal{D}}
\def\mcE{\mathcal{E}}
\def\mcF{\mathcal{F}}

\numberwithin{equation}{section}

\begin{document}
\title[self-similar Dirichlet form on pillow-type carpets]{self-similar Dirichlet form on pillow-type carpets--a short analytic construction}

\author{Shiping Cao}
\address{Department of Mathematics, The Chinese University of Hong Kong, Shatin, Hong Kong}
\email{spcao@math.cuhk.edu.hk}

\author{Hua Qiu}
\address{School of Mathematics, Nanjing University, Nanjing, 210093, P. R. China.}
\email{huaqiu@nju.edu.cn}

\author{Yizhou Wang}
\address{School of Mathematics, Nanjing University, Nanjing, 210093, P. R. China.}
\email{652022210012@smail.nju.edu.cn}

\subjclass[2010]{Primary 28A80, 31E05}
\date{}

\keywords{Sierpiński carpet, pillow-type carpets, Dirichlet forms, self-similar energies}

\thanks{The second author was supported by the National Natural Science Foundation of China (Grant No. 12471087
and 12531004).}

\begin{abstract}
We give a short, self-contained analytic proof of the existence of self-similar Dirichlet forms on pillow-type carpets, a family of infinitely ramified fractals that includes the Sierpiński carpet.   
\end{abstract}

\maketitle
\section{Introduction}
The purpose of this note is to present a short yet self-contained analytic proof of the existence of self-similar Dirichlet forms on a family of infinitely ramified fractals\textendash called pillow-type carpets\textendash\textbf{with Hausdorff dimension strictly less than 2}, which includes the classical Sierpi\'nski carpet. 

We collect some representative works on the construction of self-similar Dirichlet forms, as well as $p$-energy forms (a generalization of Dirichlet forms corresponding to $p=2$) on self-similar fractals: for post-critically finite (p.c.f.) fractals such as the Sierpi\'nski gasket, initiated by  Barlow-Perkins \cite{BP}, Kigami \cite{Kig89, Kig93}, Lindstrøm \cite{Lin}, Sabot \cite{Sab}, and Cao-Gu-Qiu \cite{CGQ}; for non-p.c.f. fractals via probabilistic approaches, pioneered by Barlow-Bass \cite{BB89, BB99}  and Kusuoka-Zhou \cite{KZ}; and for non-p.c.f. fractals via analytic approaches, recently developed by Cao-Qiu \cite{CQ}, Kigami \cite{Kig23}, Shimizu \cite{Shi}, and Murugan-Shimizu \cite{MS}. In the non-p.c.f. setting,  probabilistic approaches based on the elliptic Harnack inequality generally require strong local symmetry of the fractal, while  analytic approaches relying on Poincar\'e inequality and the Loewner condition require a suitable low-dimensional assumption. The Sierpi\'nski carpet satisfies both, making both approaches applicable. In contrast, for p.c.f. fractals, the existence problem for self-similar Dirichlet forms is typically reduced to finding specific nonnegative eigenvectors of a nonlinear Perron-Frobenius operator in a finite-dimensional space, and thus remains largely independent of the geometric inequalities central to the non-p.c.f. setting.

The analytic approach taken here follows a natural transition from the Sierpi\'nski gasket to the Sierpi\'nski carpet. We consider an  increasing sequence of graphs approximating the fractal and construct a \textbf{compatible sequence} of resistance forms along this sequence in the sense of Kigami \cite[Definition 2.2.1]{Kig01}, which ultimately yields the desired Dirichlet form on the fractal. Compared with the existing literature, this note is structured to be as accessible as possible, offering a streamlined exposition for readers primarily interested in the Sierpi\'nski carpet. The framework  used to illustrate our method\textendash the pillow-type carpets\textendash is adapted from the pillow space in \cite{DE} and will be recalled in the next section.\medskip

Throughout this note we adopt the following conventions:
\begin{enumerate}[(1)]
    \item $\mathbb{Z}_+$ denotes the set of nonnegative integers. 
    \item $|x - y|$ denotes the Euclidean distance between points $x,y\in \mathbb{R}^2$.
\end{enumerate}

\section{Pillow-type carpets}\label{s:pillow-type space}

In this section, we construct a pillow-type carpet as self-similar metric space $(F,d_{F})$ and its approaching graph sequence $V_n$. As an analogue of the celebrated Laakso space in \cite{Laa}, the pillow space was first introduced in \cite{DE}, where it was proved that once there exists a distance $\theta$ quasi-symmetric with the original distance and an Ahlfors regular measure $\mu$ such that the pillow space equipped with $(\theta,\mu)$ satisfies Poincar\'e inequality, then there is a metric measure space with analytic dimension $1$ homeomorphism to $\mathbb{R}^2$, which answers a question proposed in \cite{KS}. Other disussion on such Laakso-type spaces can be found in \cite{AES, Bar04, Mur24b}. The construction of pillow-type carpet follows by combining the construction of pillow space and Sierpi\'nski carpet together, where we will simultaneously remove small square and pile multiple small squares with same planar projection in each level.

Fix an integer $L_{F} \geq 3$ and let $Q_0:=[0,1]^2$ be the unit square in $\mathbb{R}^2$. For each $n\geq 0$, tile $Q_0$ into $L_{F}^{2n}$ subsquares of length $L_{F}^{-n}$ by
\begin{equation}\label{eq:tiling of unit square}
    \widehat{\mathcal Q}_n := \big\{[(l_1 - 1)L_{F}^{-n},l_1L_{F}^{-n}]\times [(l_2 - 1)L_{F}^{-n},l_2L_{F}^{-n}]\subset Q_0:l_i\in\mathbb{Z},i = 1,2\big\}.
\end{equation}
For each square $Q = [a_1,a_1 + r]\times [a_2,a_2 + r]\subset \mathbb{R}^2$ with $(a_1,a_2)\in\mathbb{R}^2$ and $r > 0$, we assign the orientation preserving affine map $\widehat{\Psi}_{Q}:Q_0\to Q$ by $\widehat{\Psi}_{Q}(x_1,x_2) := (rx_1 + a_1,rx_2 + a_2)$. 

Define a family of integer-valued maps $\mathscr N_0:= \big\{\nu:\widehat{\mathcal Q}_1\to \mathbb{Z}_+\big\}$ and write $\nu_{Q}:= \nu(Q)$ for each $\nu\in\mathscr N_0$ and $Q\in \widehat{\mathcal Q}_1$. For each $\nu\in \mathscr N_0$, we write the \emph{non-vanish pattern of $\nu$} by
\begin{equation}\label{eq:non-vanish pattern}
    \widehat{\mathcal Q}_{1}^{\nu}:= \big\{ Q\in\widehat{\mathcal Q}_1:\nu_{Q} \geq 1\big\}
\end{equation}
and the \emph{$1$-norm of $\nu$} by
\begin{equation}\label{eq:1-norm}
    \|\nu\| := \sum_{Q\in\widehat{\mathcal Q}_1}\nu_{Q}.
\end{equation}

\begin{definition}\label{d:admissible piling multiplicity}
    A map $\nu\in\mathscr N_0$ is called an \emph{admissible piling multiplicity}, if
    \begin{enumerate}[$(\mathrm{PC}1)$]
    \item (\emph{Symmetry}) $\nu\circ \widehat{g} = \nu$ for all isometries $\widehat{g}:Q_0\to Q_0$.
    \item (\emph{Connectedness}) $\bigcup_{Q\in\widehat{\mathcal Q}^{\nu}_1}Q$ is connected.
    \item (\emph{Flat border}) $\nu_{Q} = 1$ for all $Q\in\widehat{\mathcal Q}_1$ satisfying $Q\cap \partial Q_0 \neq\emptyset$.
    \item (\emph{Low-dimension}) The $1$-norm of $\nu$ satisfies $\|\nu\| \leq L_{F}^2 - 1$.
    \end{enumerate}
    We denote the family of all admissible piling multiplicities as $\mathscr N_{\mathrm{p}}^{\mathrm{low}}$.
\end{definition}

In this paper, we always fix an admissible piling multiplicity $\nu\in\mathscr N^{\mathrm{low}}_{\mathrm p}$. The pillow-type carpet $F$ determined by $\nu$ will be defined as an inverse limit of pre-pillow-type carpets
\begin{equation}\label{eq:inverse limit}
    F_0\stackrel{\pi_1}{\leftarrow} F_1\stackrel{\pi_2}{\leftarrow}  F_2\stackrel{\pi_3}{\leftarrow}  \cdots \stackrel{\pi_{n}}{\leftarrow}  F_{n}\stackrel{\pi_{n + 1}}{\leftarrow}  \cdots,
\end{equation}
where each $\pi_n:F_{n}\to F_{n - 1}$ is an $1$-Lipschitz map. Readers unfamiliar with pillow space may try the specific case $L_{F} = 3$ and $\nu_{Q} = 2$ for the center sequare $Q\in \widehat{\mathcal Q}_1$ at first (ignore the condition $\mathrm{(PC4)}$). Moreover, we will see that the pillow-type carpet $F$ degenerates to planar Sierpi\'nski carpet when $\nu_{Q} \leq 1$ for all $Q\in \widehat{\mathcal Q}_1$, in which case condition $(\mathrm{PC4})$ is equivalent to $\widehat{\mathcal Q}_1^{\nu}\neq \widehat{\mathcal Q}_{1}$.\vspace{.2cm}

\noindent \textbf{Symbol systems.} Introduce two \emph{alphabet sets} $S,\widehat{S}$ with respect to $\nu$ by
\[S := \big\{(Q,j)\in\widehat{\mathcal Q}^{\nu}_1\times \mathbb{Z}_+:0 \leq j \leq \nu_{Q} - 1\big\}\quad\textit{and}\quad \widehat{S} := \widehat{\mathcal Q}_{1}^{\nu},\]
and let $W_n := S^n,\widehat{W}_n := \widehat{S}^{n}$ be the \emph{set of $n$-length words} for each $n\in \mathbb{Z}_+$, where we write $W_{0} = \widehat{W}_0 = \{\emptyset\}$. It is worth noting that by $(\mathrm{PC4})$,
\begin{equation}\label{eq:counting of alphabet set}
    \# S = \|\nu\| \leq L_{F}^2 - 1.
\end{equation}
Moreover, for each $w = w_1w_2\cdots w_n\
\in W_n$ with $w_i = (Q_{w_i},j_{w_i})\in S$, we define a projection $\widehat{w}\in \widehat{W}_{n}$ by $\widehat{w} = \widehat{w}_1\widehat{w}_2\cdots \widehat{w}_n$ with $\widehat{w}_i = Q_{w_i}\in \widehat{\mathcal Q}^{\nu}_1$. For each $\widehat{w} = \widehat{w}_1\widehat{w}_2\cdots \widehat{w}_{n}\in \widehat{W}_n$, assign a square $Q_{\widehat{w}}\in \widehat{\mathcal Q}_{n}$ by
\begin{equation}
    Q_{\widehat{w}} := \widehat{\Psi}_{\widehat{w}_1}\circ\widehat{\Psi}_{\widehat{w}_2}\circ\cdots\circ \widehat{\Psi}_{\widehat{w}_n}(Q_0).
\end{equation}

\noindent\textbf{Pre-pillow-type carpets.} For each $n\in \mathbb{Z}_+$, take a collection of squares $\mathcal Q_n :=\{Q_{w}\}_{w\in W_n}$ such that each $Q_{w}$ is isometric to $Q_{\widehat{w}}$ for $\widehat{w}\in \widehat{W}_n$ by an isometry $i_{w}:Q_{w}\to Q_{\widehat{w}}$. For each $w = w_1w_2\cdots w_n\in W_n$, take $w^{-} := w_1w_2\cdots w_{n-1}\in W_{n - 1}$ and assign an isometry embedding $\pi_{n,w}:Q_{w}\to Q_{w^{-}}$ by
\begin{equation}
    \pi_{n,w} := i_{w^{-}}^{-1}\circ i_{w}.
\end{equation}

Now we define an equivalence relation $\mathcal R_n$ on the disjoint union $\bigsqcup_{w\in W_n}Q_{w}$ by induction on $n$. For $x,y\in \bigsqcup_{w\in W_0}Q_{w}$, we define $x\mathcal R_0y$ if and only if $x = y$. Assume $\mathcal R_{n - 1}$ on $\bigsqcup_{w\in W_{n-1}}Q_{w}$ has been defined. Then for each $x,y\in \bigsqcup_{w\in W_n}Q_{w}$ with $x\in Q_{w},y\in Q_{v}$ for $w,v\in W_n$, we define $x\mathcal R_{n}y$ if and only if at least one of the following cases holds.
\begin{enumerate}[(1)]
    \item $w^{-}\neq v^{-}$ and $\pi_{n,w}(x)\mathcal R_{n - 1}\pi_{n,v}(y)$.
    \item $w^{-} = v^{-}$ and $i_{w}(x) = i_{v}(y)\in \widehat{\Psi}_{\widehat{w}}(\partial Q_0)$.
    \item $x = y$.
\end{enumerate}

Now we define the \emph{$n$-level pre-pillow-type carpet} $F_n$ for each $n\in \mathbb Z_+$ as a quotient space
\[F_n := \big(\bigsqcup_{w\in W_{n}}Q_{w}\big)/\mathcal R_{n}.\]
Then each $Q_{w}\in \mathcal Q_{n}$ is naturally embedded as a subset of $F_n$ by the corresponding quotient map. Moreover, the metric $d_{F_n}$ on $F_n$ between $x,y\in F_n$ is defined by
\[d_{F_n}(x,y) := \inf\bigg\{\sum_{i = 1}^{N}|x_{i} - x_{i-1}|:\{x_i\}_{i = 0}^{N}\in \mathcal C_{n}(x,y)\bigg\},\]
where $\mathcal C_{n}(x,y)$ is the collection of all chains $\{x_i\}_{i = 0}^{N}$ satisfying $x_0 = x,x_{N} = y$ and for each $1 \leq i \leq N$, there exists $w\in W_{n}$ such that $x_{i},x_{i-1}\in Q_{w}$.

The $1$-Lipschitz maps $\pi_n:F_{n}\to F_{n-1}$ in \eqref{eq:inverse limit} is defined by combining all $\pi_{n,w}:Q_w\to Q_{w^{-}}$ together, that is
\[\pi_{n}(x) := \pi_{n,w}(x)\quad\textit{for all }w\in W_{n},x\in Q_{w}.\]
Moreover, the $1$-Lipschitz property of $\pi_n$ follows by that for all $x,y\in F_n$, it holds
\begin{equation}
    \pi_{n}(\mathcal C_n(x,y))\subset \mathcal C_{n - 1}(\pi_n(x),\pi_n(y))
\end{equation}
where $\pi_n(\mathcal C_n(x,y))$ is the image of all chains in $\mathcal C_n(x,y)$ under map $\pi_n$.

\begin{definition}\label{d:pillow-type carpet}
    Let $\{F_n\}_{n = 0}^{\infty}$ and $\{\pi_n\}_{n = 1}^{\infty}$ in \eqref{eq:inverse limit} as constructed above for a fixed admissible piling multiplicity $\nu\in\mathscr N_{\mathrm{p}}^{\mathrm{low}}$. Then we define the inverse limit metric space $(F,d_{F})$ of \eqref{eq:inverse limit} by
    \begin{equation}\label{eq:set of pillow-type carpet}
        F:= \big\{\{x_n\}_{n = 0}^{\infty}\in \prod_{n = 0}^{\infty}F_n:\pi_n(x_n) = x_{n - 1},n\geq 1\big\},
    \end{equation}
    and for each $\{x_n\}_{n = 0}^{\infty},\{y_n\}_{n = 0}^{\infty}\in F$,
    \begin{equation}\label{eq:distance of pillow-type carpet}
        d_{F}(\{x_n\}_{n = 0}^{\infty},\{y_n\}_{n = 0}^{\infty}) := \lim_{n\to\infty}d_{F_n}(x_n,y_n).
    \end{equation}
    The metric space $(F,d_{F})$ is said to be the \emph{pillow-type carpet of multiplicity $\nu$}.
\end{definition}

\noindent \textbf{Self-similarity of pillow-type carpets.} For each $n\in \mathbb{Z}_+$, we assign an $1$-Lipschitz projection $\pi_{\infty,n}:F\to F_n$ by $\pi_{\infty,n}(\{x_{m}\}_{m = 0}^{\infty}) = x_n$. Then each $w\in W_n$ corresponds to a \emph{$w$-cell} $F_{w} := \pi^{-1}_{\infty,n}(Q_{w})\subset F$. For each $x\in F$, there exists a sequence $F_{w^{(n)}}$ for $w^{(n)}\in W_n$ such that $\{x\} = \bigcap_{n = 0}^{\infty}F_{w^{(n)}}$. In particular, we write the projection from $F$ to $Q_0\subset \mathbb{R}^2$ by
\begin{equation}\label{eq:projection from carpet to plane}
    i_{\emptyset}\circ\pi_{\infty,0}(x) =:\widehat{\pi}(x) = (\widehat{\pi}_1(x),\widehat{\pi}_2(x)).
\end{equation}

To illustrate the self-similar structure of $F$, for each $w\in W_1$, we define a map $\Psi_{w}:F\to F$ by $\Psi_w(x)=y$ if and only if there exists a sequence $w^{(n)}\in W_n$ such that
\[\{x\} = \bigcap_{n = 0}^{\infty}F_{w^{(n)}}\quad\textit{and}\quad \{y\} = \bigcap _{n = 0}^{\infty}F_{ww^{(n)}},\]
where $ww^{(n)}\in W_{n + 1}$ is the conjunction of words $w$ and $w^{(n)}$. Then $F$ satisfies the self-similarity,
\[F = \bigcup_{w\in W_1}F_{w} = \bigcup_{w\in W_1} \Psi_{w}(F).\]
For general $w=w_1w_2\cdots w_n\in W_n$, write $\Psi_{w} := \Psi_{w_1}\circ \Psi_{w_2}\circ\cdots\circ \Psi_{w_n}$ as convention.\vspace{.2cm}

\noindent \textbf{Planar symmetry of pillow-type carpets.} By condition $\mathrm{(PC1)}$, each isometry $\widehat{g}:Q_0\to Q_0$ induces a map $\widehat{\tau}_{n}:\widehat{W}_n\to \widehat{W}_n$ such that $Q_{\widehat{\tau}_n(\widehat w)} = \widehat{g}(Q_{\widehat w})$ for all $\widehat{w}\in \widehat{W}_n,n\in\mathbb Z_+$, which further induces $\tau_n:W_n\to W_n$ such that for each $w = w_1w_2\cdots w_n$ with $w_i = (\widehat{w}_i,j_{w_i})$, it holds that $\tau_n(w) =: v= v_1v_2\cdots v_n$ with $\widehat{v} = \widehat{\tau}_n(\widehat{w})$ and $j_{v_i} = j_{w_i}$ for $1 \leq i \leq n$. Then we define a map $g:F\to F$ induced by $\widehat{g}$ such that for any $x\in F$, there exists a sequence $w_n\in W_n$ satisfying
\[\{x\} = \bigcap _{n = 0}^{\infty}F_{w_n}\quad\textit{and}\quad \{g(x)\} = \bigcap _{n = 0}^{\infty}F_{\tau_n(w_n)}.\]
The group of isometries $g:F\to F$ induced in this way is called the \emph{planar symmetry} of $F$, which always keeps the piling of $F$ unchanged. We will only exploit the planar symmetry of $F$ in this note.\vspace{.2cm}

Let $\widehat{p}_0 := (0,0),\widehat{p}_1 := (1,0),\widehat{p}_2 := (1,1),\widehat{p}_3 := (0,1)$ be the vertices of $Q_0$. Since $\#\widehat{\pi}^{-1}(\widehat{p}_i) = 1$, we define $p_i\in F$ such that $\widehat{\pi}^{-1}(\widehat{p}_i) = \{p_i\}$ for each $i\in\{0,1,2,3\}$.

\begin{definition}
    Define $V_0 := \{p_i\}_{i = 0}^{3}\subset F$ and say $x,y\in V_0$ are adjacent if and only if $d_{F}(x,y) = 1$. For each $n\in \mathbb{Z}_+$, define the \emph{$n$-level approaching graph} of $F$ by
    \begin{equation}
        V_n:= \bigcup_{w\in W_n}\Psi_{w}(V_0)\subset F,
    \end{equation}
    and for $x,y\in V_n$, write $x\stackrel{n}{\sim}y$ if and only if there exists adjacent $x',y'\in V_0$ and $w\in W_n$ such that $x = \Psi_w(x')$ and $y = \Psi_{w}(y')$.
\end{definition}

Each $V_{n}$ is also called the \emph{$n$-level $0$-skeleton} of $F$. Combining condition $\mathrm{(PC2)}$ and $\mathrm{(PC3)}$, each graph $V_{n}$ is connected. Moreover, the following uniformly locally finite property of sequence $V_n$ is mainly attributed to condition $\mathrm{(PC3)}$.

\begin{lemma}
    There is a constant $c_{\mathrm{ulf}} \geq 1$ only depending on $L_{F}$ such that
    \begin{equation}\label{eq:metric doubling property}
        \#\big\{w\in W_m:x\in F_{w}\big\}\leq c_{\mathrm{ulf}}\quad\textit{for all }m\geq 0,x\in F,
    \end{equation}
    \begin{equation}\label{eq:locally finite property}
        \#\big\{y\in V_m:x\stackrel{m}{\sim}y\big\} \leq c_{\mathrm{ulf}}\quad\textit{for all }m\geq0,x\in V_m.
    \end{equation}
\end{lemma}
\begin{proof}
    Notice that for each $x\in V_{m},m\geq 0$, we have 
    \begin{equation}
        \#\big\{y\in V_m:x\stackrel{m}{\sim}y\big\} \leq 4\#\big\{w\in W_m:x\in F_{w}\big\},
    \end{equation}
    which reduces the proof of \eqref{eq:locally finite property} to \eqref{eq:metric doubling property}. Define the $n$-level $1$-skeleton of $F$ by
    \[S_n := \bigcup_{Q\in \widehat{\mathcal Q}_n}\widehat{\pi}^{-1}(\widehat{\Psi}_{Q}(\partial Q_0))\subset F,\]
    and denote $S_{-1} := \emptyset$. If $x\in F\setminus S_m$, then \eqref{eq:metric doubling property} holds with right hand side taking $1$. It remains to prove \eqref{eq:metric doubling property} for $x\in S_m$. 
    
    To this end, for each $0 \leq n \leq m$, we let 
    \[
    N_{m,n}=\sup_{x\in S_{n}\setminus S_{n - 1}}\#\big\{w\in W_m:x\in F_{w}\big\}
    \]
    By condition $\mathrm{(PC3)}$, $N_{m,0} \leq 2$ for all $m\geq 0$. If $x\in S_{m}\setminus S_{m-1}$, then there exists a unique $v\in W_{m - 1}$ such that $x\in F_{v}$, which implies $N_{m,m} = N_{1,1} \leq \|\nu\|$ by pulling back $F_{v}$ to $F$ with $\Psi_{v}^{-1}$. More generally, for each $0 \leq n \leq m$ and $x\in S_{n}\setminus S_{n - 1}$,
    \[
    \#\{w\in W_{m}:x\in F_{w}\} \leq N_{m - n,0}\cdot\#\{w\in W_{n}:x\in F_{w}\} \leq N_{m - n,0}N_{n,n} \leq 2\|\nu\|,\]
    which implies the desired estimate \eqref{eq:metric doubling property} by $(\mathrm{PC}4)$.
\end{proof}

\section{Resistance estimate on graph sequence}
For each $n\geq 0$, we write $l(V_n) := \mathbb{R}^{V_n}$ as the family of all functions with domain $V_n$. A functional $D:l(V_n)\times l(V_n)\to \mathbb{R}$ is said to be a \emph{graph energy} on $V_n$, if there exists $c:V_n\times V_n\to [0,\infty)$ such that
\[D(f,g) = \sum_{x\stackrel{n}{\sim}y}c(x,y)\big(f(x) - f(y)\big)\big(g(x) - g(y)\big),\]
where in the summation, we omit the requirement $x,y\in V_n$ for short. Also, we write $D(f) := D(f,f)$. For each $0 \leq n \leq m$ and graph energy $D$ on $l(V_m)$, we denote the trace of $D$ on $V_n$ as
\begin{equation}
    [D]_{V_n}(f) := \min\big\{D(g):g\in l(V_m),g|_{V_n} = f\big\}
\end{equation}
for all $f\in l(V_n)$. Each functional $[D]_{V_n}$ gives a graph energy on $V_n$ by polarization.

For each $n\geq 0$, define the \emph{natural graph energy} on $V_n$ by
\begin{equation}\label{eq:graph energy}
    \mathcal D_n(f,g) := \sum_{x\stackrel{n}{\sim}y}\big(f(x) - f(y)\big)\big(g(x) - g(y)\big)\quad\textit{for }f,g\in l(V_n).
\end{equation}
For each $0 \leq n \leq m$, decomposing $V_{m}$ into multiple subgraphs isometric to $V_{m - n}$ with respect to graph distance each enveloped by $n$-level skeleton of $F$, we have the following \emph{$(m,n)$-reduction of natural graph energy} 
\begin{equation}\label{eq:reduction of graph energy}
    c_{\mathrm{ulf}}^{-1}\sum_{w\in W_{n}}\mathcal D_{m - n}(f\circ \Psi_{w}) \leq \mathcal D_{m}(f) \leq \sum_{w\in W_{n}}\mathcal D_{m - n}(f\circ \Psi_{w}),
\end{equation}
where the left hands follows by that each pair of adjacent $x,y\in V_{m}$ is at most occupied by $c_{\mathrm{ulf}}$ many $F_{w},w\in W_{n}$ simultaneously by \eqref{eq:metric doubling property}.

For each disjoint subsets $A,B\subset V_n$, we define the \emph{effective resistance} between $A,B$ by
\begin{equation}\label{eq:resistance of disjoint sets}
    R_n(A,B) := \big(\min\big\{\mathcal D_n(f):f\in l(V_n),f|_{A} \equiv 1,f|_{B}\equiv 0\big\}\big)^{-1}.
\end{equation}
For each distinct $x,y\in V_n$, write $R_n(x,y)$ shortly for $R_n(\{x\},\{y\})$. The goal of this section is to prove the multiplicity estimate for sequence
\begin{equation}\label{eq:resistance between vertices}
    \mathscr R_{n} := R_n(p_0,p_1).
\end{equation}
Notice that by the planar symmetry of $F$, the pair $p_0,p_1$ in \eqref{eq:resistance between vertices} can be replaced by any other adjacent pair in $V_0$.

\begin{theorem}\label{t:multiplicity of resistance constant constants}
    There exists a constant $c_{\dagger} \geq 1$ depending on $L_{F}$ so that for all $0 \leq n \leq m$,
    \begin{equation}\label{eq:multiplicity of resistance constants}
        c_{\dagger}^{-1}\mathscr R_{n}\mathscr R_{m - n} \leq \mathscr R_m \leq c_{\dagger}\mathscr R_{n}\mathscr R_{m - n}.
    \end{equation}
\end{theorem}

We take two steps to complete the proof of Theorem \ref{t:multiplicity of resistance constant constants}.

\subsection{A trace estimate.} For each $0 \leq n \leq m$, we will prove a trace theorem comparing $[\mathcal D_m]_{V_n}$ and $\mathcal D_{n}$, see Proposition \ref{p:trace theorem}. For the extension part of the trace theorem, the resistance between border lines will be used, that is
\begin{equation}\label{eq:resistance between border lines}
    \overline{\mathscr R}_n:= R_{n}(\overline{p_3p_0}\cap V_{n},\overline{p_1p_2}\cap V_n).
\end{equation}
Here $\overline{p_ip_j}\subset F$ is defined by $\overline{p_ip_j} := \widehat{\pi}^{-1}(\overline{\widehat p_i\widehat p_j})$, where for each $x,y\in \mathbb R^2$, $\overline{xy}$ is the line segment with end points $x,y$. By the planar symmetry of $F$, the definition of $\overline{\mathscr R}_n$ is unchanged when taking the other pair of opposite borders.

\begin{proposition}\label{p:trace theorem}
    For each $0 \leq n \leq m$, there exist constants $c_{\mathrm{r}}, c_{\mathrm{e}} \geq 1$ depending on $L_{F}$ such that for all $f\in l(V_n)$,
    \begin{equation}\label{eq:trace theorem}
        c_{\mathrm e}^{-1}\overline{\mathscr R}_{m - n}[\mathcal D_m]_{V_n}(f)\leq \mathcal D_{n}(f) \leq c_{\mathrm r}\mathscr R_{m - n}[\mathcal D_{m}]_{V_n}(f).
    \end{equation}
    Specifically, we have
    \begin{enumerate}[(1)]
        \item For each $f\in l(V_m)$, we have
        \begin{equation}\label{eq:restriction theorem}
            \mathcal D_{n}(f|_{V_n}) \leq c_{\mathrm{r}}\mathscr R_{m - n}\mathcal D_{m}(f).
        \end{equation}
        \item There exists an extension operator $\mathfrak{E}_{n,m}:l(V_{n})\to l(V_{m})$ such that $(\mathfrak{E}_{n,m}f)|_{V_{n}} = f$,
        \begin{equation}\label{eq:local perturbation}
            \min_{V_0}f\circ\Psi_{w} = \min_{V_{m - n}}\mathfrak{E}_{n,m}f\circ \Psi_{w}\quad\textit{and}\quad \max_{V_0}f\circ \Psi_{w} = \max_{V_{m - n}}\mathfrak{E}_{n,m}f\circ \Psi_{w}
        \end{equation}
        for all $w\in W_{n}$, and
        \begin{equation}\label{eq:extension theorem}
            \overline{\mathscr R}_{m - n}\mathcal D_m(\mathfrak{E}_{n,m}f) \leq c_{\mathrm{e}}\mathcal D_{n}(f).
        \end{equation}
    \end{enumerate}
\end{proposition}
\begin{proof}
    (1) For each $w\in W_n$, 
    \begin{equation*}
    \begin{split} 
    \mcD_0\big((f|_{V_n})\circ\Psi_w\big)&= \sum_{x\stackrel{0}{\sim}y}\big(f\circ\Psi_w(x)-f\circ\Psi_w(y)\big)^2\\
    &\leq \sum_{x\stackrel{0}{\sim}y}\mathscr{R}_{m-n}\mathcal D_{m-n}(f\circ\Psi_w)=4\mathscr{R}_{m-n}\mathcal D_{m-n}(f\circ\Psi_w). 
    \end{split} 
    \end{equation*}
   The inequality then follows by applying $(n,n)$-reduction and $(m,n)$-reduction in \eqref{eq:reduction of graph energy}.\smallskip 

    (2) We first construct four \emph{pre-extension kernels} $\psi_i\in l(V_{m - n})$ for $0 \leq i \leq 3$ such that
    \begin{enumerate}[$(\mathrm{PEK}1)$]
        \item (Partition of unity) $0 \leq \psi_i \leq 1$ and $\psi_0 + \psi_1 +\psi_2 + \psi_3 \equiv 1$.
        \item (Boundary condition) $\psi_i(p_i) = 1$ and $\psi_i|_{\overline{p_jp_k}} \equiv 0$ for two pairs of adjacent $p_j,p_k\in V_0\setminus\{p_i\}$. 
        \item (Planar symmetry) The value of $\psi_i$ on the two unconditioned borders are symmetric to each other. For each pair $\psi_i,\psi_j$, each one is a planar symmetry of the other one.
        \item (Energy estimate) $\mathcal D_{m - n}(\psi_i) \leq c_{\mathrm{pek}}\overline{\mathscr R}_{m - n}^{-1}$ for some constant $c_{\mathrm{pek}} \geq 1$.
    \end{enumerate}

    To this end, we take $h\in l(V_{m - n})$ as the minimizer of the variation in the definition of $\overline{\mathscr R}_{m - n}$, that is $h|_{\overline{p_3p_0}}\equiv 1$, $h_{\overline{p_1p_2}}\equiv 0$ and $\mathcal D_{m - n}(h) = \overline{\mathscr R}_{m - n}^{-1}$. It is worth noting that additionally we have $h|_{\widehat{\pi}^{-1}([0,1/2]\times [0,1])}\geq 1/2$ by the planar symmetry exchanging $\overline{p_3p_0}$ and $\overline{p_1p_2}$. Indeed, this is clear when $L_{F}$ is an even number by the maximum principle of harmonic functions, while the case $L_{F}$ is odd follows by adding auxiliary points in $V_{m - n}$ along the corresponding symmetry axis.

    A planar symmetry transformation of $h$ induces a function $h'\in l(V_{m - n})$ such that $h'|_{\overline{p_0p_1}}\equiv 1$,  $h'|_{\overline{p_2p_3}}\equiv 0$ and $\mathcal D_{m - n}(h') = \overline{\mathscr R}_{m - n}^{-1}$. Define
    \[\widetilde{\psi}_0 := (2h)\wedge (2h')\wedge 1,\]
    which satisfies $\widetilde{\psi}_0|_{\widehat{\pi}^{-1}([0,1/2]\times [0,1/2])} \equiv 1$, $\widetilde{\psi}_0|_{\overline{p_1p_2}\cup \overline{p_2p_3}} \equiv 0$ and
    \begin{equation}\label{eq:energy of pre-pre-extension kernel}
        \mathcal D_{m - n}(\widetilde{\psi}_0) \leq \mathcal D_{m - n}(2h) + \mathcal D_{m - n}(2h') = 8\overline{\mathscr R}_{m- n}^{-1}.
    \end{equation}
    Define $\widetilde{\psi}_1,\widetilde{\psi}_2,\widetilde{\psi}_3\in l(V_{m - n})$ as planar symmetry transformation of $\widetilde{\psi}_0$ such that $\widetilde{\psi}_i(p_i) = 1$ for $1 \leq i \leq 3$, and let
    \begin{equation}\label{eq:pre-extension kernel}
        \widetilde{\psi} := \sum_{i = 0}^{3}\widetilde{\psi}_i\quad\textit{and}\quad \psi_i := \widetilde\psi_i/\widetilde{\psi}\quad\textit{for } 0 \leq i \leq 3.
    \end{equation}
    We check that $\psi_i$ satisfies the condition $(\mathrm{PEK1})$-$(\mathrm{PEK4})$.

    Condition $(\mathrm{PEK1})$-$(\mathrm{PEK3})$ follow from \eqref{eq:pre-extension kernel} and the boundary value of each $\widetilde{\psi}_i$. To prove condition ($\mathrm{PEK4}$), notice that $0 \leq \widetilde{\psi}_i \leq 1$ and  $1\leq \widetilde{\psi} \leq 4$ on $F$, which implies
    \[\begin{aligned}
        \mathcal D_{m - n}(\psi_i) = \mathcal D_{m - n}(\widetilde{\psi}_i\cdot \widetilde{\psi}^{-1}) & \leq 2\|\widetilde{\psi}^{-1}\|_{\infty}^2\mathcal D_{m - n}(\widetilde{\psi}_i) + 2\|\widetilde{\psi}_i\|_{\infty}^2\mathcal D_{m - n}(\widetilde{\psi}^{-1})\\
        & \leq 2\|\widetilde{\psi}^{-1}\|_{\infty}^2\mathcal D_{m - n}(\widetilde{\psi}_i) + 2\|\widetilde{\psi}_i\|_{\infty}^2\|\widetilde{\psi}^{-1}
        \|_{\infty}^2\mathcal D_{m - n}(\widetilde{\psi})\\
        & \leq 2\mathcal D_{m - n}(\widetilde{\psi}_i) + 2\mathcal D_{m - n}(\widetilde{\psi})\\
        & \leq (2 + 2\cdot 4^2)\mathcal D_{m - n}(\widetilde{\psi}_0)\\
        & = 272\overline{\mathscr R}_{m - n}^{-1},
    \end{aligned}\]
    where $\|\cdot\|_{\infty}$ denotes the maximum norm of functions, and the last line follows by \eqref{eq:energy of pre-pre-extension kernel}. Thus all conditions $(\mathrm{PEK1})$-$(\mathrm{PEK4})$ are satisfied by $\psi_i$ with constant $c_{\mathrm{pek}} := 272$.\vspace{.2cm}

    For each $u\in l(V_0)$, we define the \emph{pre-extension operator} $\mathfrak{E}_{0,m - n}u\in l(V_{m - n})$ by
    \begin{equation}
        \mathfrak{E}_{0,m - n}u := \sum_{i = 0}^{3}u(p_i)\psi_i,
    \end{equation}
    which satisfies $\mathfrak{E}_{0,m - n}u|_{V_0} = u$ by condition ($\mathrm{PEK2}$), and that
    \begin{enumerate}[$(\mathrm{PE}1)$]
        \item The value of $\mathfrak{E}_{0,m-n}u$ on $\overline{p_ip_j}$ is independent of the value of $u$ on $p_{k},p_l\in V_0\setminus \{p_i,p_j\}$ by $(\mathrm{PEK2})$.
        \item The relation between the value of $\mathfrak{E}_{0,m-n}u$ on $\overline{p_ip_j}$ and the value of $u$ on $p_i,p_j$ is independent of the choice of $i,j$ by $(\mathrm{PEK3})$.
        \item $\min_{V_0}u \leq \mathfrak{E}_{0,m - n}u(x) \leq \max_{V_0} u$ for all $x\in V_{m - n}$ by $(\mathrm{PEK1})$.
        \item It has energy upper bound
    \begin{equation}
    \begin{aligned}
        \mathcal D_{m - n}(\mathfrak{E}_{0,m - n}u) 
        & = \mathcal D_{m - n}\big(\sum_{i = 0}^{3}(u(p_i) - c_u)\psi_i\big) \leq 4 \mathcal D_{m - n}(\psi_0)\sum_{i = 0}^{3}\big(u(p_i) - c_u\big)^2\\
        & \leq 4c_{\mathrm{pek}}\overline{\mathscr R}_{m - n}^{-1}\times 2\mathcal D_0(u) = 8c_{\mathrm{pek}}\overline{\mathscr R}^{-1}_{m - n}\mathcal D_0(u),
    \end{aligned}
    \end{equation}
    \end{enumerate}
    where the constant $c_u := \big(\sum_{i = 0}^{3}u(p_i)\big)/4$ is average of $u$, the first equation follows by $(\mathrm{PEK1})$, and the second inequality follows by ($\mathrm{PEK}4$).

    Now for each $u\in l(V_n)$, we define the \emph{extension operator} $\mathfrak{E}_{n,m}u\in l(V_m)$ satisfying that for all $w\in W_{n}$,
    \begin{equation}
        (\mathfrak{E}_{n,m}u)\circ \Psi_{w} = \mathfrak{E}_{0,m - n}(u\circ \Psi_{w}),
    \end{equation}
    which is compatible on each $\Psi_{w}(V_{m - n})\cap \Psi_{v}(V_{m - n})$ by $(\mathrm{PE1})$ and $(\mathrm{PE2})$. Moreover, \eqref{eq:local perturbation} follows by $(\mathrm{PE3})$, and the desired estimate \eqref{eq:extension theorem} follows by applying $(m,n)$-reduction and $(n,n)$-reduction in \eqref{eq:reduction of graph energy}, together with $\mathrm{(PE4)}$, on each sides of \eqref{eq:extension theorem}. This completes the proof.
\end{proof}

As a consequence of Proposition \ref{p:trace theorem}, we have the following weak version of Theorem \ref{t:multiplicity of resistance constant constants}. Notice that since the feasible family of variation $\mathscr R_{n}$ is larger than $\overline{\mathscr R}_n$, we have a priori estimate
\begin{equation}\label{eq:a priori comparing between resistance constants}
    \overline{\mathscr R}_n \leq \mathscr R_n.
\end{equation}

\begin{corollary}\label{c:weak multiplicity}
    For all $0 \leq n \leq m$ and constant $c_{\mathrm{r}},c_{\mathrm e}\geq 1$ in Proposition \ref{p:trace theorem}, we have
    \begin{equation}\label{eq:weak multiplicity}
        c^{-1}_{\mathrm{e}}\mathscr R_{n}\overline{\mathscr R}_{m - n} \leq \mathscr R_{m} \leq c_{\mathrm r}\mathscr R_{n}\mathscr R_{m - n}.
    \end{equation}
\end{corollary}
\begin{proof}
    For each $k\geq 0$, take $h_k\in l(V_k)$ as the minimizer of the variation $\mathscr R_{k}$. Then by Proposition \ref{p:trace theorem} (1), we have
    \[\mathscr R_{n}^{-1} \leq \mathcal D_{n}(h_{m}|_{V_n}) \leq c_{\mathrm{r}}\mathscr R_{m - n}\mathcal D_{m}(h_m) = c_{\mathrm{r}}\mathscr R_{m - n}\mathscr R^{-1}_{m},\]
    which implies the right hand side of \eqref{eq:weak multiplicity}. On the other hand, by Proposition \ref{p:trace theorem} (2), we have
    \begin{equation}
        \mathscr R^{-1}_{m} \leq \mathcal D_{m}(\mathfrak E_{n,m}h_n) \leq c_{\mathrm{e}}\overline{\mathscr R}_{m - n}^{-1}\mathcal D_{n}(h_n) = c_{\mathrm{e}}\overline{\mathscr R}_{m - n}^{-1}\mathscr R^{-1}_{n},
    \end{equation}
    which implies the left hand side of \eqref{eq:weak multiplicity}.
\end{proof}

Taking $m = n +1$ in Corollary \ref{c:weak multiplicity}, we derive the two-sided doubling property of sequence $\mathscr R_{n}$, that is, there exists a constant $c_{\mathrm{d}}\geq 1$ independent of $n$ such that
\begin{equation}\label{eq:two-sides doubling or resistance constants}
    c^{-1}_{\mathrm d}\mathscr R_{n} \leq \mathscr R_{n + 1} \leq c_{\mathrm d}\mathscr R_{n}.
\end{equation}

\subsection{Comparison between resistance constants.} To upgrade Corollary \ref{c:weak multiplicity} to Theorem \ref{t:multiplicity of resistance constant constants}, it remains to prove the converse estimate of \eqref{eq:a priori comparing between resistance constants}.

Due to the condition $\mathrm{(PC4)}$ of admissible piling multiplicity in Definition \ref{d:admissible piling multiplicity}, we have the following a priori geometric increasing lower bound of $\overline{\mathscr R}_n$ and reversed doubling property of sequence $\mathscr R_{n}$.

\begin{lemma}\label{l:a priori lower bound estimate of resistance constants}
    There exist constants $c_{\mathrm{a}} \geq 1$ and $\rho_{\mathrm a}\in(1,\infty)$ depending on $L_{F}$ such that for all $n\geq 0$, we have
    \begin{equation}\label{eq:a priori lower bound estimate of resistance constants}
        \overline{\mathscr R}_n \geq c^{-1}_{\mathrm a}\rho_{\mathrm{a}}^{n}.
    \end{equation}
    In particular, for each $0 \leq n \leq m$, we have the reversed doubling estimate with constant $c_{\mathrm{rd}} > 0$ depending on $L_{F}$,
    \begin{equation}\label{eq:reversed doubling estimate}
        \rho_{\mathrm{a}}^{m - n}\mathscr R_{n} \leq c_{\mathrm{rd}}\mathscr R_{m}.
    \end{equation}
\end{lemma}
\begin{proof}
    The estimate \eqref{eq:reversed doubling estimate} is reduced into \eqref{eq:a priori lower bound estimate of resistance constants} by the left hand side of Corollary \ref{c:weak multiplicity}, thus it remains to prove \eqref{eq:a priori lower bound estimate of resistance constants}. Let $h\in C(F)$ defined by $h(x) := 1 - \widehat{\pi}_1(x)\in[0,1]$, where $\widehat{\pi}_1$ is defined in \eqref{eq:projection from carpet to plane}. Then $h|_{\overline{p_3p_0}}\equiv 1$, $h|_{\overline{p_1p_2}}\equiv 0$ and
    \[\overline{\mathscr R}_{n}^{-1} \leq \mathcal D_n(h|_{V_n}) = \sum_{x\stackrel{n}{\sim}y}(h(x) - h(y))^2 \leq 4c_{\mathrm{ulf}}\|\nu\|^{n}L_{F}^{-2n} \leq 4c_{\mathrm{ulf}}(L_{F}^2 - 1)^nL_{F}^{-2n},\]
    where the second inequality follows by that the number of edges in $V_n$ is not larger than $c_{\mathrm{ulf}}\# V_n \leq 4c_{\mathrm{ulf}}\|\nu\|^n$. The desired result follows by taking $c_{\mathrm{a}} := 4c_{\mathrm{ulf}}$, $\rho_{\mathrm{a}} := L^2_{F}/(L_{F}^2 - 1)$.
\end{proof}

The estimate \eqref{eq:reversed doubling estimate} is called reversed doubling estimate, since it is equivalent to the existence of $N_0\geq 1$ such that for all $n\geq 0$,
\[2\mathscr R_{n} \leq \mathscr R_{n + N_0},\]
which provides a tension against two-sides doubling estimate \eqref{eq:two-sides doubling or resistance constants}. 

\begin{corollary}\label{c:Holder continuity on graph}
    For each $f\in l(V_m)$ and $x,y\in V_m$, we have
    \begin{equation}\label{eq:Holder continuity on graph}
        |f(x) - f(y)|^2 \leq c_{\mathrm{H}}d_{F}(x,y)^{\theta_{\mathrm{H}}}\mathscr R_{m}\mathcal D_{m}(f),
    \end{equation}
    where $c_{\mathrm H} \geq 1$ and $0 < \theta_{\mathrm{H}} < 2$ are constants depending on $L_{F}$.
\end{corollary}
\begin{proof}
    If $x,y\in V_n$ is adjacent in $V_n$ for some $0 \leq n \leq m$, then $d_{F}(x,y) = L_{F}^{-n}$, and there exists $w\in W_n$ such that $x,y\in \Psi_{w}(V_0)$. Then by \eqref{eq:reversed doubling estimate}, we have
    \begin{equation}\label{eq:Holder continuity for adjacent points}
        |f(x) - f(y)|^2 \leq \mathscr R_{m - n}\mathcal D_{m - n}(f
        \circ \Psi_{w}) \leq c_{\mathrm{rd}}\rho_{\mathrm a}^{-n}\mathscr R_{m}\mathcal D_{m}(f).
    \end{equation}

    For general $x\neq y\in V_{m}$, there exists a chain $x_i$ for $0 \leq i \leq N$ such that 
    \begin{enumerate}[(1)]
        \item $x_0 = x,x_{N} = y$.
        \item For each $1 \leq i \leq N$, there exists $M_{x,y} \leq n \leq m$ such that $x_i,x_{i-1}$ is adjacent in $V_{n}$, where $0 \leq M_{xy} \leq m$ is the smallest integer such that $-\log d_{F}(x,y)/\log L_{F} \leq M_{xy}$.
        \item For each $M_{x,y} \leq n \leq m$, it holds that
        \begin{equation}\label{eq:chain argument}
            \#\big\{1 \leq i \leq N:x_{i}\stackrel{n}{\sim}x_{i - 1}\big\} \leq c_{\mathrm{chain}},
        \end{equation}
        where $c_{\mathrm{chain}}\geq 1$ only depending on $L_{F}$.
    \end{enumerate}
    Then we have
    \[\begin{aligned}
        |f(x) - f(y)| \leq \sum_{i = 1}^{N}|f(x_i) - f(x_{i-1})| & \leq \sum_{n = M_{x,y}}^{m}c_{\mathrm{chain}}\cdot c_{\mathrm{rd}}^{1/2}\rho_{\mathrm a}^{-n/2}\mathscr R_{m}^{1/2}\mathcal D_{m}(f)^{1/2}\\
        & \leq c_{\mathrm{H}}^{1/2}d_{F}(x,y)^{\theta_{\mathrm{H}}/2}\mathscr R_{m}^{1/2}\mathcal D_{m}(f)^{1/2},
        \end{aligned}\]
    where the second inequality follows by \eqref{eq:Holder continuity for adjacent points} and \eqref{eq:chain argument}, and $\theta_{H}:= \log \rho_{\mathrm a}/\log L_{F} < 2$.
\end{proof}

Now we are at the stage to prove the main target of this subsection.

\begin{proposition}\label{p:converse estimate between resistance constants}
    There exists a constant $c_{\mathrm{c}} \geq 1$ such that $\mathscr R_{m} \leq c_{\mathrm{c}}\overline{\mathscr R}_m$.
\end{proposition}
\begin{proof}
    Let $h\in l(V_m)$ be the minimizer of variation $\mathscr R_{m}$, that is $h(p_0) = 1,h(p_1) = 0$ and $\mathscr R_{m}\mathcal D_{m}(h) = 1$. Then by Corollary \ref{c:Holder continuity on graph}, there exists a constant $M\geq 0$ such that for all $m\geq M$ and $w,v\in W_{M}$ such that $p_0\in F_{w},p_1\in F_{v}$, it holds that
    \[h|_{\Psi_{w}(V_{m - M})} \geq  3/4\quad\textit{and}\quad h|_{\Psi_{v}(V_{m - M})} \leq 1/4,\]
    since the diameter of $\Psi_{w}(V_{m - M}),\Psi_{v}(V_{m - M})$ with respect to $d_{F}$ is approximately $L_{F}^{-M}$ independent of $m$.

    We define the subgraph $V'_{m}\subset V_{m}$ by
    \begin{equation}\label{eq:definition of V'}
        V'_{m} := \bigcup\big\{\Psi_{u}(V_{m - M}):u\in W_{M},F_{u}\cap \overline{p_0p_1}\neq\emptyset\big\},
    \end{equation}
    and restrict our discussion from graph $V_{m}$ to $V'_{m}$. Define graph energy $\mathcal D'_{m}$ on $l(V'_m)$ by replacing the range of $x,y$ from $V_{m}$ to $V'_{m}$ in \eqref{eq:graph energy}, and define the effective resistance $R'_{m}$ on $V'_{m}$ by replacing $\mathcal D_m$ with $\mathcal D'_m$ in \eqref{eq:resistance of disjoint sets}. Then we have
    \begin{equation}\label{eq:reduce to bold border}
        \mathscr R_m^{-1} = \mathcal D_{m}(h) \geq \mathcal D'_m(h) \geq \mathcal D'_m(\overline h) \geq 4^{-1}\big(R'_{m}(F_{w}\cap V'_m,F_{v}\cap V'_m)\big)^{-1},
    \end{equation}
    where $\overline{h} := (h\wedge (3/4))\vee (1/4)$. 

    We label the words in \eqref{eq:definition of V'} by $w_i\in W_{M}$ for $1 \leq i \leq L_{F}^M =:N$ such that $F_{w_i}\cap \overline{p_0p_1}\neq\emptyset$, $w_1 = w,w_{N} = v$ and $F_{w_i}\cap F_{w_{i-1}}\neq\emptyset$. Define $r_i := R'_{m}(F_{w}\cap V'_m,F_{w_i}\cap V'_m)$ for each $3 \leq i \leq N$, which satisfies $r_{N} \geq \mathscr R_{m}/4$ by \eqref{eq:reduce to bold border}, and $r_3 = \overline{\mathscr R}_{m- M}$.
    
    For each $4 \leq i \leq N$, we claim the following doubling property
    \begin{equation}\label{eq:doubling reduction along bold border}
        r_{j} \geq r_i/4\quad\textit{for all } i/2 + 1\leq j \leq i.
    \end{equation}
    Indeed, take $h_i\in l(V'_m)$ as the minimizer of variation $r_i$, which satisfies $h_i|_{\Psi_{w_j}(V_{m - M})} \leq 1/2$ for all $i/2 + 1 \leq j \leq i$ by left-right symmetry along $V'_m$. Then the desired doubling property \eqref{eq:doubling reduction along bold border} follows by taking function $(2h_i - 1)\vee 0$, which is feasible for the variation $r_j$. 
    
    Iterating \eqref{eq:doubling reduction along bold border} with $j = i - 1$ as $i$ varying from $N$ to $4$, it implies 
    \[\overline{\mathscr R}_{m -M} = r_{3} \geq 4^{-N}r_{N} \geq 4^{-N - 1}\mathscr R_m  \geq 4^{-N -1}c_{\mathrm d}^{-M}\mathscr R_{m - M}\]
    for all $m\geq M$, where the last inequality follows by \eqref{eq:two-sides doubling or resistance constants}. Thus the desired result follows by taking $c_{\mathrm c} := 4^{N + 1}c_{\mathrm d}^{M}$.
\end{proof}

\noindent \emph{Proof of Theorem \ref{t:multiplicity of resistance constant constants}. } 
It follows by combining Corollary \ref{c:weak multiplicity} and Proposition \ref{p:converse estimate between resistance constants}.\hfill$\qed$\vspace{.2cm}

Recall that every sequence $a_n > 0$ satisfying $a_{n + m} \asymp a_{n}a_{m}$ for all $n,m\geq 0$ admits a ratio $r > 0$ such that $a_n \asymp r^{n}$ by Fekete's lemma. By Theorem \ref{t:multiplicity of resistance constant constants} and Lemma \ref{l:a priori lower bound estimate of resistance constants}, there exists $r_{\nu} > 1$ such that 
\begin{equation}\label{eq:fekete's lemma}
    c^{-1}_{\mathrm F}r_{\nu}^n \leq \mathscr R_{n}\leq c_{\mathrm F}r_{\nu}^n
\end{equation}
for all $n\geq 0$, where $c_{\mathrm F} \geq 1$ is determined by Feteke's lemma. 

\section{Self-similar Dirichlet forms on pillow-type carpets}

First we convert the natural graph energy sequence $\mathcal D_n$ into a compatible sequence $\widetilde{\mathcal D}_n$.

\begin{lemma}\label{l:decimation graph energy}
    There exists a sequence of graph energies $\widetilde{\mathcal D}_n$ on $V_n$ satisfying
    \begin{equation}\label{eq:decimation energy vs natural energy}
        c_{\mathrm D}^{-1}\mathscr R_{n}\mathcal D_n(f) \leq \widetilde{\mathcal D}_n(f) \leq c_{\mathrm D}\mathscr R_{n}\mathcal D_n(f)
    \end{equation}
    for all $f\in l(V_n)$, where $c_{\mathrm D} \geq 1$ is a constant depending on $L_{F}$, and for all $0 \leq n \leq m$, it satisfies decimation formula
    \begin{equation}\label{eq:decimation formula}
        [\widetilde{\mathcal D}_m]_{V_n} = \widetilde{\mathcal D}_n.
    \end{equation}
\end{lemma}
\begin{proof}
    First notice that combining with Theorem \ref{t:multiplicity of resistance constant constants} and Proposition \ref{p:converse estimate between resistance constants}, the trace theorem Proposition \ref{p:trace theorem} can be rephrased as
    \begin{equation}\label{eq:rephrased trace theorem}
        c_{\mathrm t}^{-1}[\mathscr R_m \mathcal D_m]_{V_n}\leq \mathscr R_n\mathcal D_n \leq c_{\mathrm t}[\mathscr R_m\mathcal D_m]_{V_n},
    \end{equation}
    where the constant $c_{\mathrm t}:= c_{\mathrm r}c_{\mathrm e}c_{\mathrm \dagger}c_{\mathrm{c}}\geq 1$ only depending on $L_{F}$. 

    Then for each fixed $n\geq 0$, the sequence $[\mathscr R_m\mathcal D_m]_{V_n}$ for $m\geq n$ admits a subsequence $m_k$ for $k\geq 0$ of $[\mathscr R_{m}\mathcal D_m]_{V_n}$ locally uniformly converging to a limit functional, denoted as $\widetilde{\mathcal D}_n$, on $l(V_n)$ by Arzel\`a-Ascoli theorem (see \cite[Lemma 4.3]{CGQ} for a similar argument)
    , which satisfies the desired \eqref{eq:decimation energy vs natural energy} by \eqref{eq:rephrased trace theorem}. Here $\widetilde{\mathcal D}_n$ is still a graph energy on $V_n$ by polarization. 
    
    Moreover, we additionally assume that the subsequence $m_k$ is common for all $n\geq 0$ by an extra diagonal argument. The desired \eqref{eq:decimation formula} follows by
    \[\widetilde{\mathcal D}_n = \lim_{k\to\infty}[\mathscr R_{m_k}\mathcal D_{m_k}]_{V_n} = \lim_{k\to\infty}[[\mathscr R_{m_k}\mathcal D_{m_k}]_{V_m}]_{V_n} = [\lim_{k\to\infty}[\mathscr R_{m_k}\mathcal D_{m_k}]_{V_m}]_{V_n} = [\widetilde{\mathcal D}_m]_{V_n}.\]
    This completes the proof.
\end{proof}

We fix a reference measure $\mu$ on the pillow-type carpet $F$ such that every ball charges finite positive measure of $\mu$.

\begin{definition}
    Define a functional $\widetilde{\mathcal E}:C(F)\to [0,\infty)$ and a subspace $\mathcal F\subset C(F)$ by
    \begin{equation}
        \widetilde{\mathcal E}(f) := \lim_{n\to\infty}\widetilde{\mathcal D}_n(f|_{V_n})\quad\textit{and}\quad \mathcal F:= \big\{f\in C(F):\widetilde{\mathcal E}(f) < \infty\big\}.
    \end{equation}
    Moreover, we take polarization of $\widetilde{\mathcal E}$ and write $\widetilde{\mathcal E}_1(f,g) := \widetilde{\mathcal E}(f,g) + (f,g)_{L^2(F,\mu)}$.
\end{definition}
\begin{remark}\label{r:Holder continuity of domain of dirichlet form}
    \begin{enumerate}[(1)]
        \item Denote $V_*:= \bigcup _{n = 0}^{\infty}V_n$, which is a dense subset of $F$. For each $f\in \mathcal F$ and $x,y\in V_n$, by Corollary \ref{c:Holder continuity on graph}, we have
        \begin{equation}\label{eq:Holder continuity on F}
            |f(x) - f(y)|^2 \leq c_{\mathrm{H}}c_{\mathrm D}d_{F}(x,y)^{\theta_{\mathrm H}}\widetilde{\mathcal D}_n(f|_{V_n}) \leq c_{\mathrm{H}}c_{\mathrm D}d_{F}(x,y)^{\theta_{\mathrm H}}\widetilde{\mathcal E}(f).
        \end{equation}
        This implies the continuous embedding $\mathcal F\subset C^{\theta_{\mathrm H}/2}(F)$, the space of $(\theta_{\mathrm H}/2)$-H\"older continuous functions.
        \item If $f_i\in\mathcal F,i\geq 1$ is a sequence such that $\widetilde{\mathcal E}_1(f_i) \leq M$ for some $M > 0$, then the boundedness of $\widetilde{\mathcal E}(f_i)$ implies the uniform continuity of $f_i$ by \eqref{eq:Holder continuity on F}. Combining this with the boundedness of $\|\cdot\|_{L^2(F,\mu)}$, it implies that $f_i$ is uniformly bounded in $C(F)$. Then by Arzel\`a-Ascoli theorem, $f_i$ admits a uniformly convergent subsequence with a continuous function limit.
        \item Recall the extension operator $\mathfrak{E}_{n,m}:l(V_n)\to l(V_m)$ in Proposition \ref{p:trace theorem}, which satisfies \eqref{eq:local perturbation} and
        \begin{equation}
            \widetilde{\mathcal D}_m(\mathfrak E_{n,m}f) \leq c_{\mathfrak E}\widetilde{\mathcal D}_{n}(f)
        \end{equation}
        for all $f\in l(V_n)$, where $c_{\mathfrak E} := c_{\mathrm e}c_{\mathrm c}c_{\dagger}c_{\mathrm D}^2 \geq 1$ only depending on $L_{F}$. As an advantage of $\mathfrak{E}_{n,m}$, the condition \eqref{eq:local perturbation} may not be satisfied by the minimizer of $[\widetilde{\mathcal D}_{m}]_{V_n}(f)$ in $l(V_m)$.
    \end{enumerate}
\end{remark}

Indeed, $(V_n,\widetilde{\mathcal D}_n)$ gives a compatible sequence in \cite[Definition 2.2.1]{Kig01}, thus the following lemma mainly follows by \cite[Theorem 2.2.6, Theorem 2.4.1, Theorem 3.3.4]{Kig01}. For the convenience of readers, we provide a self-contained proof.

\begin{lemma}\label{l:properties of closed form}
    $(\widetilde{\mathcal E},\mathcal F)$ is a  regular Dirichlet form on $L^2(F,\mu)$, which satisfies 
    \begin{equation}\label{eq:quasi-self-similarity}
        c_{\mathrm{qs}}^{-1}\sum_{w\in W_{n}}r_{\nu}^n\widetilde{\mathcal E}(f\circ \Psi_w) \leq \widetilde{\mathcal E}(f) \leq c_{\mathrm{qs}}\sum_{w\in W_{n}}r_{\nu}^n\widetilde{\mathcal E}(f\circ \Psi_w)
    \end{equation}
    for all $n\geq 0$ as a quasi-self-similarity, where $r_{\nu}$ is given by \eqref{eq:fekete's lemma} and $c_{\mathrm{qs}}\geq 1$ only depending on $L_{F}$. Moreover, $(\mathcal F,\widetilde{\mathcal E}_1)$ is a separable Hilbert space.
\end{lemma}
\begin{proof}
    First, we prove the closed property of $(\widetilde{\mathcal E},\mathcal F)$. Take a $\widetilde{\mathcal E}_1$-Cauchy sequence $f_i\in \mathcal F$. It admits $f\in L^2(F,\mu)$ such that $f_i$ converges weakly to $f$ in $L^2(F,\mu)$. Moreover, by Remark \ref{r:Holder continuity of domain of dirichlet form} (2), any subsequence of $f_i$ admits a sub-subsequence uniformly converging to a continuous version of $f$, which implies $f\in C(F)$ and $f_i$ uniformly converges to $f$. It remains to check that $\lim_{i\to\infty}\widetilde{\mathcal E}(f_i - f) = 0$, which indeed follows by
    \[\begin{aligned}
    \limsup_{i\to\infty}\widetilde{\mathcal E}(f_i - f) 
    &= \limsup_{i\to\infty}\sup_{n\geq 0}\widetilde{\mathcal D}_n(f_i|_{V_n} - f|_{V_n}) = \limsup_{i\to\infty}\sup_{n\geq 0}\lim_{j\to\infty}\widetilde{\mathcal D}_n(f_i|_{V_n} - f_j|_{V_n}) \\
    & \leq \limsup_{i\to\infty}\sup_{n\geq 0}\limsup_{j\to\infty}\widetilde{\mathcal E}(f_i - f_j) = \limsup_{i\to\infty}\limsup_{j\to\infty}\widetilde{\mathcal E}(f_i - f_j) = 0,
    \end{aligned}\]
    where the second equation follows by that $f_i(x)$ converges to $f(x)$ for all $x\in V_*$. 

    Next, we prove that $\mathcal F$ is dense in $C(F)$. Take $f\in C(F)$ and denote
    \[\omega_{n}(f):=\sup\{|f(x) - f(y)|:x,y\in F_{w}\textit{ for some }w\in W_n\},\]
    which satisfies $\lim_{n\to\infty}\omega_n(f) = 0$. For each $0 \leq n \leq m$ and fixed $f\in C(F)$, define the Mcshane extension of $\mathfrak{E}_{n,m}(f|_{V_n})\in l(V_m)$ by
    \[f_{n,m}(x) := \inf_{y\in V_m}\big\{\mathfrak{E}_{n,m}(f|_{V_n})(y) + c^{1/2}_{
    \mathrm H}c_{\mathrm D}^{1/2}d_{F}(x,y)^{\theta_{\mathrm H}/2}\widetilde{\mathcal D}_m(\mathfrak{E}_{n,m}(f|_{V_n}))^{1/2}\big\}\]
    for $x\in F$, which satisfies $f_{n,m}|_{V_m} = \mathfrak{E}_{n,m}(f|_{V_n})$ by \eqref{eq:Holder continuity on F} and 
    \[\begin{aligned}
        |f_{n,m}(x) - f_{n,m}(y)| 
        & \leq \sup_{z\in V_m}c^{1/2}_{\mathrm H}c_{\mathrm D}^{1/2}(d_{F}(x,z)^{\theta_{\mathrm H}/2} - d_{F}(y,z)^{\theta_{H}/2})\widetilde{\mathcal D}_m(\mathfrak{E}_{n,m}(f|_{V_n}))^{1/2}\\
        & \leq c_{\mathrm H}^{1/2}c_{\mathrm D}^{1/2}c_{\mathfrak E}^{1/2}d_{F}(x,y)^{\theta_{\mathrm{H}}/2}\widetilde{\mathcal D}_n(f|_{V_n})^{1/2}
    \end{aligned}\]
    for all $x,y\in F$ and $m\geq n$, where the second inequality follows by Remark \ref{r:Holder continuity of domain of dirichlet form} (3). Then by Arzela-Ascoli theorem, there exists a subsequence of $f_{n,m}$ for $m\geq n$ uniformly converging to a continuous function limit, denoted as $f_{n,m_{k}}\to f_{n}\in C(F)$. Moreover, we have
    \[\begin{aligned}
        \widetilde{\mathcal E}(f_n) & = \lim_{m\to\infty}\widetilde{\mathcal D}_m(f_n|_{V_m}) = \lim_{m\to\infty}\lim_{k\to\infty}\widetilde{\mathcal D}_m(f_{n,m_k}|_{V_m}) = \lim_{m\to\infty}\lim_{k\to\infty}\widetilde{\mathcal D}_m(\mathfrak{E}_{n,m_k}(f|_{V_n})|_{V_m}) \\
        & \leq \lim_{m\to\infty}\lim_{k\to\infty}\widetilde{\mathcal D}_{m_k}(\mathfrak{E}_{n,m_k}(f|_{V_n})) = \lim_{k\to\infty}\widetilde{\mathcal D}_{m_k}(\mathfrak{E}_{n,m_k}(f|_{V_n})) \leq  c_{\mathfrak E}\widetilde{\mathcal D}_{n}(f|_{V_n}) < \infty,
    \end{aligned}\]
    which implies $f_n\in\mathcal F$. It remains to prove $f_n$ uniformly converges to $f$. 
    
    Notice that $f_n|_{V_n} = f$ for all $n\geq 0$. Recall that by \eqref{eq:local perturbation}, for each $w\in W_n$ and $x\in \Psi_{w}(V_*)$, we have 
    \[\begin{aligned}
        |f_n(x) - f_n(\Psi_w(p_0))| 
        & = \lim_{k\to\infty}|f_{n,m_k}(x) - f_{n,m_k}(\Psi_w(p_0))| \\
        & = \lim_{k\to\infty}|\mathfrak{E}_{n,m_k}(f|_{V_n})(x) - \mathfrak{E}_{n,m_k}(f|_{V_n})(\Psi_w(p_0))| \\
        & \leq \max_{1 \leq i \leq 3}|f(\Psi_w(p_i)) - f(\Psi_{w}(p_0))| \leq \omega_n(f),
    \end{aligned}\]
    which further implies
    \[\begin{aligned}
        |f_n(x) - f(x)| \leq |f_n(x) - f_n(\Psi_w(p_0))| + |f(\Psi_w(p_0)) - f(x)| \leq 2\omega_n(f),
    \end{aligned}\]
    thus $f_n$ uniformly converges to $f$.

    Till now, we have seen that $(\widetilde{\mathcal E},\mathcal F)$ is a densely defined, closed, regular symmetric form on $L^2(F,\mu)$, where the densely defined property follows by that $C(F)$ is dense in $L^2(F,\mu)$. Moreover, $(\widetilde{\mathcal E},\mathcal F)$ also satisfies the Markovian property, which follows by the corresponding property of the graph energy sequence $\widetilde{\mathcal D}_n$. Thus $(\widetilde{\mathcal E},\mathcal F)$ is a regular Dirichlet form.

    The quasi-self-similarity \eqref{eq:quasi-self-similarity} follows by combining Lemma \ref{l:decimation graph energy}, \eqref{eq:fekete's lemma} and the $(m,n)$-reduction of graph energy \eqref{eq:reduction of graph energy} for all $m\geq n\geq 0$.

    To prove that $(\mathcal F,\widetilde{\mathcal E}_1)$ is separable, take a countable dense subset $E_n\subset l(V_n)$ for each $n\geq 0$ and let
    \begin{equation}
        E:=\bigcup_{n\geq 0}\big\{\mathfrak H_{n}f\in\mathcal F:f\in E_n\big\},
    \end{equation}
    where $\mathfrak{H}_n f$ is the minimizer of $\min\big\{\widetilde{\mathcal E}(g):g\in\mathcal F,g|_{V_n} = f\big\}$. Then $E$ gives a countable dense subset of $(\mathcal F,\widetilde{\mathcal E}_1)$.
\end{proof}

\begin{theorem}\label{t:self-similar dirichlet form}
        There is a strongly local regular Dirichlet form $(\mathcal E,\mathcal F)$ on $L^2(F,\mu)$ satisfying the self-similarity that
        \begin{equation}\label{eq:self-similarity of domain}
            \mathcal F = \big\{f\in C(F):f\circ\Psi_{w}\in\mathcal F, w\in W_1\big\},
        \end{equation}
        and for $f\in \mathcal F$,
        \begin{equation}\label{eq:self-similarity of functional}
            \mathcal E(f) = r_{\nu}\sum_{w\in W_1}\mathcal E(f\circ\Psi_w).
        \end{equation}
        Moreover, there exists a constant $c_{\mathrm S} \geq 1$ only depending on $L_{F}$ such that
        \begin{equation}\label{eq:compatible between self-similar form and graph energy sequence}
            c^{-1}_{\mathrm S}\lim_{n\to\infty}\widetilde{\mathcal D}_n(f|_{V_n}) \leq \mathcal E(f)\leq c_{\mathrm S}\lim_{n\to\infty}\widetilde{\mathcal D}_n(f|_{V_n}).
        \end{equation} 
\end{theorem}

\begin{proof}
First, we prove \eqref{eq:self-similarity of domain}. For each $f\in C(F)$, by \eqref{eq:quasi-self-similarity} in Lemma \ref{l:properties of closed form}, we have
    \[f\in\mathcal F\Leftrightarrow \widetilde{\mathcal E}(f) < \infty\Leftrightarrow \sup_{w\in W_1}\widetilde{\mathcal E}(f\circ \Psi_{w}) <\infty\Leftrightarrow f\circ\Psi_{w}\in\mathcal F\textit{ for all }w\in W_1.\]

Next, we construct the functional $\mathcal E$ satisfying \eqref{eq:compatible between self-similar form and graph energy sequence}. Since $(\mathcal F,\widetilde{\mathcal E}_1)$ is separable, take a countable dense subset $E\subset \mathcal F$ and expand $E$ to
\[E_{\infty}:= \big\{f\circ\Psi_{w}\in\mathcal F:f\in E,w\in W_n\textit{ for some }n\geq 0\big\},\]
which is a dense countable subset of $(\mathcal F,\widetilde{\mathcal E}_1)$ satisfying that $f\in E_{\infty}$ implies $f\circ\Psi_{w}\in E_{\infty}$ for all $w\in W_n,n\geq 0$. Then by \eqref{eq:quasi-self-similarity}, there exists a sequence $m_k$ such that
\[\mathcal E(f) := \lim_{k\to\infty}\frac{1}{m_k}\sum_{n = 1}^{m_k}\sum_{w\in W_{n}}r_{\nu}^{n}\widetilde{\mathcal E}(f\circ\Psi_{w})\]
converges for all $f\in E_{\infty}$, which satisfies
\begin{equation}\label{eq:comparing of self-similar form and quasi-self-similar closed form}
    c_{\mathrm{qs}}^{-1}\widetilde{\mathcal E}(f) \leq \mathcal E(f) \leq c_{\mathrm{qs}}\widetilde{\mathcal E}(f)
\end{equation}
for all $f\in E_{\infty}$, thus admitting a unique continuous extension on $(\mathcal F,\widetilde{\mathcal E}_1)$, still denoted as $\mathcal E$, which further satisfies \eqref{eq:comparing of self-similar form and quasi-self-similar closed form} for all $f\in\mathcal F$. Then $(\mathcal E,\mathcal F)$ is a regular Dirichlet form on $L^2(F,\mu)$, by the Markovian property of each $\widetilde{\mathcal E}(f\circ\Psi_{w})$ for $w\in W_n,n\geq 0$. The desired \eqref{eq:compatible between self-similar form and graph energy sequence} follows by \eqref{eq:comparing of self-similar form and quasi-self-similar closed form}.

Next, we prove \eqref{eq:self-similarity of functional}. It is enough to prove \eqref{eq:self-similarity of functional} for $f\in E_{\infty}$, which follows by
\[\begin{aligned}
    \bigg|\mathcal E(f) - \sum_{w\in W_1}r_{\nu}\mathcal E(f\circ\Psi_{w})\bigg| 
    & = \lim_{k\to\infty}\bigg|\frac{1}{m_k}\sum_{n = 1}^{m_k}\sum_{w\in W_{n}}r_{\nu}^n\widetilde{\mathcal E}(f\circ\Psi_w) - \frac{1}{m_k}\sum_{n = 2}^{m_{k}+1}\sum_{w\in W_{n}}r_{\nu}^n\widetilde{\mathcal E}(f\circ\Psi_w)\bigg|\\
    & = \lim_{k\to\infty}\frac{1}{m_k}\bigg|\sum_{w\in W_1}r_{\nu}\widetilde{\mathcal E}(f\circ\Psi_w) - \sum_{w\in W_{m_k + 1}}r_{\nu}^{m_k + 1}\widetilde{\mathcal E}(f\circ\Psi_w)\bigg|\\
    & \leq \lim_{k\to\infty}2c_{\mathrm{qs}}\widetilde{\mathcal E}(f)/m_k= 0.
\end{aligned}\]

Finally, for $f,g\in \mcF$ so that $f$ is a constant on a neighbor of the support of $g$, we can find $n$ large enough so that either $f|_{F_w}$ or $g|_{F_w}$ is a constant for every $w\in W_n$. Then, by  \eqref{eq:self-similarity of functional}
\[
\mcE(f,g)=r_{\nu}^n\sum_{w\in W_n}\mcE(f\circ\Psi_w,g\circ\Psi_w)=0.
\]
This proves the strongly local property.
\end{proof}

\section*{Acknowledgments}
We are grateful to Mathav Murugan for introducing us the model of pillow space. \bibliographystyle{amsplain}

\end{document}